\documentclass [a4paper,12pt]{article}
\usepackage{amsmath}
\usepackage{amsfonts}
\usepackage{indentfirst}
\usepackage{esint}
\usepackage{latexsym}
\usepackage{bbm}
\usepackage{amsfonts}
\usepackage{mathrsfs}
\usepackage{amssymb}
\usepackage{amsmath}
\usepackage{amsthm}
\usepackage{latexsym}
\usepackage{indentfirst}
\usepackage{enumitem}
\usepackage{bm}
\usepackage{esint}
\usepackage{hyperref}
\usepackage{cleveref}
\usepackage[numbers,square]{natbib}
\numberwithin{equation}{section}
\usepackage{color}
\allowdisplaybreaks
\usepackage{amsthm}
\usepackage{amssymb}
\usepackage{enumerate}
\usepackage{ulem}
\allowdisplaybreaks

\newtheorem{theorem}{Theorem}[section]
\newtheorem{lemma}[theorem]{Lemma}

\newtheorem{defn}[theorem]{Definition}
\newtheorem{remark}[theorem]{Remark}

\makeatletter
\usepackage{amsmath}
\usepackage{amsfonts}
\usepackage{indentfirst}
\usepackage{esint}
\usepackage{latexsym}
\usepackage{authblk}
\usepackage{color}
\UseRawInputEncoding
\usepackage{amsthm}
\usepackage{amssymb}
\usepackage{enumerate}
\usepackage{ulem}
\makeatletter

\newcommand{\Rmnum}[1]{\expandafter\@slowromancap\romannumeral #1@}
\makeatother
\begin{document}

\title{The lifespan of positive solutions of heat equation with power-logarithmic nonlinearity on locally finite graph}
\author[a,b]{Pengxiu Yu\thanks{Email: Pxyu@ruc.edu.cn}}
\author[c,d]{Yiping Zhang\thanks{Corresponding author, Email: zhangyiping161@mails.ucas.ac.cn}}
\affil[a]{\footnotesize{School of Mathematics,
Renmin University of China, Beijing 100872, China}}
\affil[b]{\footnotesize{Fakult\"{a}t f\"{u}r Mathematik, Universit\"{a}t Bielefeld, Bielefeld 100131, Germany
}}
\affil[c]{\footnotesize{School of Mathematics and Statistics, and Hubei Key Laboratory of Mathematical Sciences, Central China Normal University, Wuhan 430079, China}}
\affil[d]{\footnotesize{Key Laboratory of Nonlinear Analysis \& Applications (Ministry of Education), Central China Normal University, Wuhan 430079, China}}
\date{}
\maketitle
\begin{abstract}
On a locally finite connected graph $G=(V,E)$,
using the first eigenvalue method introduced by Kaplan \cite{MR160044} and the discrete
Phragm\'{e}n-Lindel\"{o}f principle developed by Hu-Wang \cite{cvhuyuanyang}, we first establish the asymptotic behaviour of the lifespan of positive solutions to a semilinear heat equation
with the power-logarithmic nonlinearity $u^p|\log u|^q$, provided that the initial datum is bounded below by a positive constant. These results extend those of Hu-Wang \cite{cvhuyuanyang} to equations with a power-logarithmic source term.  Moreover, by means of a more direct argument, we show that analogous lifespan estimates remain valid for nonnegative initial datum $u(x,0)$, provided that $u(x_i,0)$ is suitably large at some vertex $x_i\in V$.



\end{abstract}

\section{Introduction and main results}\label{s1}
\noindent
Let $G=G(V,E,w,\mu)$ be a locally finite connected weighted graph and $\Delta$ be the usual graph Laplacian defined by
\begin{equation*}
\Delta u(x)=\frac{1}{\mu(x)}
\sum_{x,y\in V, y\sim x}w_{xy}(u(y)-u(x)).
\end{equation*}
Please refer to the text preceding Theorem \ref{thm1} for the precise meanings of the symbols appearing in the definition of $\Delta$.

\noindent
Recently, Hu-Wang \cite{cvhuyuanyang} investigated
the following Cauchy problem
\begin{equation}\label{cvequ}
\left\{
\begin{array}{lll}
u_t=\Delta u +u^p, \ \text{ in }\ V\times(0,T),\\[10pt]
u(x,0)=\lambda\psi(x),\ \ \text{for } x\in V,
\end{array}
\right.
\end{equation}
where the parameters $p>1$, $\lambda>0$, $T>0$
and
\begin{equation*}
\psi: V \to [0, \infty) \text{ is a bounded function, but not identically zero.}
\end{equation*}
Under some conditions, 
the lifespan $T_\lambda$ (the maximal existence time) of the positive solution $u$ of \eqref{cvequ} is investigated and satisfies the estimates
\begin{equation}\label{cvlim}
\lim_{\lambda\to\infty}
\lambda^{p-1}T_\lambda
=\frac{1}{(p-1)\|\psi\|^{p-1}_{l^{\infty}(V)}}.
\end{equation}

It is then natural to ask what happens when
$p=1$ in the estimates \eqref{cvlim}. By classical parabolic theory, for finite-energy initial datum, the solution to \eqref{cvequ} exists globally in time and can not blow up in finite time. Any unbounded growth, if it occurs, can only occur as $t\rightarrow\infty$. Therefore, to investigate the borderline regime between linear growth and power-type nonlinearities, we consider the logarithmically superlinear nonlinearity $u(\log u)^q$ $(q>1)$
in the equation. More exactly, we consider the following equation
\begin{equation}\label{equ}
\left\{
\begin{array}{lll}
u_t-\Delta u =u^p\left|\log u\right|^q,\ \ \text{in}\ V\times(0,T),\\[8pt]
u(x,0)=\lambda\psi(x),\ \ \text{for } x\in V,
\end{array}
\right.
\end{equation}
where
$p=1$ and $q>1$ or $p>1$ and $q\geq 1$.
Here, $\lambda>0$, $T>0$, and $\psi$ are as in \eqref{cvequ}.
It is easy to see that, when $p=q=1$, the spatially homogeneous
function
\[
u(x,t)=e^{e^t}
\]
solves the equation with the initial datum $u(x,0)=e$. This solution
grows double-exponentially in time and exhibits infinite-time
grow-up. Therefore, we only consider the cases $p=1$ with $q>1$, and $p>1$ with $q\geq 1$.

There are many literature concerning the properties of solutions to a semilinear parabolic equation. Namely, for some $T>0$ and a nontrivial, nonnegative, bounded and continuous function $\psi$, consider the following Cauchy problem
\begin{equation}\label{Cauchy}
\left\{
\begin{aligned}
&u_t=\Delta u+u^p
\quad\text{in}\quad\mathbb{R}^N\times(0,T),\\[8pt]
&u(x,0)=\psi(x)\quad\text{for} \quad x\in\mathbb{R}^N.
\end{aligned}
\right.
\end{equation}

\noindent
The maximal existence time of the solution to \eqref{Cauchy} is defined by
$$T[\psi]:=\sup\bigl\{
T>0:\ \eqref{Cauchy} \text{ admits a unique non-negative classical solution
in } \mathbb{R}^N\times(0,T)\bigr\}.$$

\noindent
There is a large body of literature concerning the existence and
nonexistence of global solutions to problem \eqref{Cauchy}. The first behavior of
solutions is characterized by the celebrated Fujita phenomenon. By using the heat-kernel representation, the comparison principle, and nonlinear integral estimates,
Fujita \cite{MR214914} proved that, if
$p>1+\frac{2}{N},$
then problem \eqref{Cauchy} admits a global classical solution for
sufficiently small initial datum, whereas every nontrivial nonnegative
solution blows up in finite time if $1<p<1+\frac{2}{N}$.

For the critical cases $N=1$ with $p=3$ and $N=2$ with $p=2$, based on refined lower-bound estimates for the Duhamel formula and an iterative procedure produces logarithmic growth, Hayakawa \cite{MR338569} established the nonexistence of nontrivial global solutions. Moreover, Kobayashi-Sirao-Tanaka \cite{MR450783} generalized and systematized Hayakawa’s iterative lower-bound method to arbitrary dimensions and more general nonlinearities that includes the critical power $p=1+\frac{2}{N}$. 

Now, we turn to the estimates of the lifespan of solutions to equation \eqref{Cauchy}.
Under the assumption
$\liminf_{|x|\to+\infty}\phi(x)>0$,
by combining comparison arguments and heat-kernel estimates to reduce the problem to suitable ordinary differential inequalities,
Lee and Ni \cite{MR1057781} proved that $T[\lambda\phi]<\infty$
for every $\lambda>0$ and the asymptotic behaviour 
$C_1\lambda^{1-p}\leq T[\lambda\phi]\leq C_2\lambda^{1-p}$, for small $\lambda>0$.

Later, based on a parabolic rescaling, comparison with the ordinary differential equation $V'=V^p$ and Kaplan’s first-eigenfunction method on large balls, Gui and Wang \cite{MR1308611}  refined the Lee-Ni estimates \cite{MR1057781} by determining the exact leading-order constants, provided that $\phi(x)\rightarrow\phi_\infty>0$ as $|x|\rightarrow\infty$. A comprehensive account of semilinear parabolic equations posed on $\mathbb{R}^N$ can be found, for example,
in \cite{MR1651764,MR1742850,MR1056055} and the references therein.

%

Notably, equation \eqref{Cauchy} has also been extensively studied on manifolds, such as the $N$-dimensional hyperbolic space $\mathbb{H}^N$. We refer readers interested in this topic to \cite{MR2823663,MR3569152} and the references therein for further details.

While semilinear equations involving the pure-power nonlinearity $u^p$
have been investigated extensively, a natural question is how the introduction of a logarithmic perturbation in the reaction term affects the qualitative properties of their solutions. In this direction, $a\in\mathbb{R}$ and $p$  
belonging to the Sobolev subcritical regime, Hamza and Zaag \cite{MR4393386} considered the semilinear heat equation
\begin{equation}\label{arch}
\partial_tu-\Delta u
=|u|^{p-1}u\bigl(\log(2+u^2)\bigr)^a\qquad
\text{in } \mathbb{R}^N\times[0,T),
\end{equation}
and determined the precise blow-up rate, agreeing with the rate predicted by the associated ordinary differential equation. In particular, every blow-up solution in the Sobolev subcritical range exhibits type-I blow-up.

Chabi and Souplet~\cite{MR4865249} extended the theory to nonlinear source term
of the form $u^pL(u)$,
focusing on the case where $L$ is a slowly varying function,
including logarithms.
Specially, they obtained a precise
description of the global blow-up behavior for radially symmetric decreasing solutions. Related results and topics can be
found in~\cite{MR4892727,MR3852471}.

%

The study of nonlinear Cauchy problems \eqref{Cauchy} has gradually been extended from Euclidean spaces to more general settings, such as graphs~\cite{MR4581147}, metric measure spaces~\cite{MR2981021}, and Riemannian manifolds~\cite{MR4340784}. In these frameworks, the global behavior of solutions is strongly influenced by the geometry of the ambient space. Quantities such as volume growth, heat-kernel decay, spectral bounds, and estimates for the graph distance often replace the role played by dimension and scaling in the classical theory. Discrete equations on graphs are especially noteworthy, since they arise naturally in models of diffusion on networks and, at the same time, exhibit geometric features that are absent from continuous domains.

Several contributions have addressed nonlinear equations on weighted or infinite graphs. For the power nonlinearity, the nonexistence result in~\cite{MR5007569} was obtained by combining graph-adapted cutoff functions with a test-function argument and a discrete integration formula. Finite-time blow-up on locally finite graphs was analyzed in~\cite{MR3794184} through auxiliary integral quantities satisfying suitable ordinary differential inequalities. In~\cite{MR5037980}, both global solvability and blow-up for~\eqref{Cauchy} were studied by using heat-semigroup estimates, spectral information for the graph Laplacian, the Duhamel representation, and supersolution constructions. Related elliptic problems have also received attention. Yamabe-type equations were treated in~\cite{MR3542963,MR4397680} by variational and finite-domain approximation methods, whereas Schr\"odinger equations in~\cite{MR4733143,MR4568177} were investigated through constrained variational principles, Nehari-type arguments, and compactness techniques.

We next focus on the problem \eqref{equ}, which is the main subject of this paper. First, 
for the reader's convenience, we recall some basic notation on graphs that will be used throughout the paper. Denote a graph
 $G=(V,E)$ with the vertex set $V$ and the edge set $E$. For any $x,y\in V$, $y\sim x$ means $xy\in E$. Let $\mu: V\rightarrow \mathbb{R}^+$ be a finite measure and we assume positive
symmetric weights $w_{xy}=w_{yx}$ on edges
$xy\in E$. For any $x\in V$,  $deg(x)$ is  the degree of $x$, denoting  the number of neighbors of $x$. If $deg(x)<\infty$, then $G=(V,E)$ is called
a locally finite graph. In our paper, we always assume that $G=(V,E)$ is an infinite locally finite connected graph without loops or multiple edges.

For a measure $\mu$ and a weight $w_{xy}$, we define
\begin{equation*}
D_{\mu}:=\sup\limits_{x\in V}\frac{\sum\limits_{x\in V}w_{xy}(x)}{\mu(x)}.
\end{equation*}
The integral of function $u(x)$
is denoted by
\begin{equation*}
\int_V u(x) d\mu
:=\sum_{x\in  V}\mu(x)u(x)
\end{equation*}
and for interval $I\subset\mathbb{R}$ and integer $n\geq0$, we define
\begin{equation*}
C_V^n(I):=
\{
u:V\times I\rightarrow\mathbb{R}:u(x,\cdot)\in C^n(I)\quad\text{for each}\quad x\in V
\},
\end{equation*}
and
\begin{equation*}
L_V^1(I):=
\{
u:V\times I\rightarrow\mathbb{R}:u(x,\cdot)\in L^1(I)\quad\text{for each}\quad x\in V
\}.
\end{equation*}
For a function $u:V\times[0,T)\rightarrow\mathbb{R}$, if the term $u_t$ appear, then we implicitly assume that $u\in C_V^1(0,T)$.
We say
$u=u(x,t;\lambda\phi):
V\times[0,T)\rightarrow\mathbb{R}^+$  is a solution of \eqref{equ} in $[0,T)$, if $u(x,\cdot)\in C^1_V(0,T)\cap C_V^0[0,T)$
satisfies \eqref{equ} and
$u\in L^{\infty}(V\times[0,T'])$ for any $0<T'<T$. The lifespan (maximal existence time) of
$u=u(x,t;\lambda\phi)$ is defined by
$$T_\lambda:=\sup\left\{T>0:\eqref{equ}\text{ admits a solution on }[0,T)
\text{ with initial datum }\lambda\phi\right\}.$$
If $T_\lambda<\infty$, then $T_\lambda$ is called the blow-up time.

Let $\Omega$ be a finite subset in $V$. Using the similar notations in \cite{cvhuyuanyang}, the mapping  $\pi_{\Omega}:l^2(V,\mu)\rightarrow l^2(\Omega,\mu)$ is a restriction operator. For $u\in l^2(V,\mu)$, we have
$\pi_{\Omega}u\in l^2(\Omega,\mu)$.
The mapping $\xi_{\Omega}:l^2(\Omega,\mu)\rightarrow l^2(V,\mu)$ is the canonical inclusion extending
$u\in l^2(\Omega,\mu)$ by zero outside $\Omega$.
Now we define $\Delta_{\Omega}:=\pi_{\Omega}\Delta\xi_{\Omega}$.
The self-adjointness of $\Delta$ implies the  self-adjointness of $\Delta_{\Omega}$ and $\pi_{\Omega}=\xi^*_{\Omega}$, then the corresponding
eigenvalues are $0\leq\lambda_1\leq\lambda_2
\leq\cdots\leq\cdots\lambda_m$ and $m=\#\Omega$.
Furthermore, when $\Omega$ is connected and $\partial(V\backslash\Omega)\neq\emptyset$, then $\lambda_1(\Omega)>0$, with $\lambda_1(\Omega)<\lambda_2(\Omega)$ and we can choose eigenfunction $\phi(x)>0$ in $\Omega$ associated with $\lambda_1(\Omega)$  and $\sum\limits_{x\in\Omega}\mu(x)\phi(x)=1$ after normalizing.

Throughout this paper, we always assume  $\Omega\subset V$ is finite connected and the boundary of $\Omega$ is $\partial\Omega=\{x\in\Omega:\ \text{there exits}\ y\in\Omega^c\ \text{such that}\ y\sim x
\}$. We also use the notation $C(\alpha,\beta)$ to mean a constant depending only on $\alpha$ and $\beta$. Now, we state our main results.

\begin{theorem}\label{thm1}
Let
$p=1$, $q>1$ and $u$ be a positive solution to \eqref{equ}. Assume that
\[
D_{\mu}<\infty,\qquad
\eta(0)\ge1,
\]
where $\eta(0)$ is defined in \eqref{1***}. Then there exists a constant
$C=C(q,V)>0$ such that
$T_\lambda<\infty$ whenever $\lambda>\Lambda$ and $\Lambda$ satisfies
\begin{equation*}
\Lambda\cdot\inf_{x\in V}\psi\geq \max\{ C(q,V ),e^{1-q} \}.
\end{equation*}
Furthermore,
$$
\lim_{\lambda\to +\infty}
\big(\log\left(\lambda\|\psi\|_{\infty}\right)\big)^{q-1}T_\lambda
=\frac{1}{q-1}.
$$
\end{theorem}

Note that, by \eqref{1***}, the conditions
$\Lambda\inf\limits_{x\in V}\psi(x)\geq e^{1-q}$
and $\eta(0)\geq1$ automatically hold whenever
$\Lambda\inf\limits_{x\in V}\psi(x)\geq 1$.

In Theorem \ref{thm1}, we denote $\inf\psi=\inf_{x\in V}\psi(x) $ and
require that $\Lambda\inf\psi\ge e^{1-q}$, because we need to use
Jensen's inequality and the function $f(x):=x|\log x|^q$ is convex for $x\geq e^{1-q}$. Now if $\inf \psi=0$ (which should be viewed as $\inf\psi\rightarrow 0^+$ such that $\inf\psi\left|\log (\inf\psi)\right|^q=0$), then the following result shows that $u(x_i,t)$ will blow up in finite time if $\lambda \psi(x_i)$ is suitably large.
\begin{theorem}\label{thm2}
Let $p=1, q>1$ and $u$ be a positive solution to \eqref{equ}. Suppose that
\[
D_\mu<\infty,\quad \inf_{x\in V}\psi\geq 0, \quad\deg(x_i)=m\geq 1
\] for some $x_i\in V$ such that $\lambda\psi(x_i)>C(V)\exp\left\{m^{1/q}\right\} $, then the positive solution $u$ to problem
\eqref{equ} (actually, $u(x_i,t)$) blows up in finite time. Furthermore, we have
$$\lim\limits_{\lambda\rightarrow+\infty}(\log(\lambda\psi(x_i)))^{q-1}T_{\lambda}
\leq\displaystyle\frac{1}{q-1}\leq \lim\limits_{\lambda\rightarrow+\infty}(\log(\lambda\|\psi\|_{\infty}))^{q-1}T_{\lambda}.$$
\end{theorem}
\begin{remark}
The statements and results in Theorem \ref{thm2} are different from those in Theorem \ref{thm1}. Roughly speaking, in Theorem \ref{thm1}, one focuses on the blow-up time globally in $V$, i.e., one tries to find the $u(x_j,t)$ with the earliest blow-up time for some $x_j\in V$ (usually, $u(x_j,t)$ is the largest one for all $x\in V$). While in Theorem \ref{thm2}, one focuses on the blow-up time locally in $V$, i.e., if one can find some $x_i\in V$ such that the initial datum $u(x_i,0)$ is suitably large, then $u(x_i,t)$ will blow up in finite time and the blow-up time of $u$ is certainly less than or equal to this of $u(x_i,t)$.
\end{remark}

Similar to the two results above, for $p>1$ and $q\geq 1$, we have the following results.
\begin{theorem}\label{thm4}
Suppose $p>1$, $q\geq1$ and $u$ is a positive solution to \eqref{equ}. If
\begin{equation*}
D_\mu<\infty,\quad \inf_{x\in V}\psi\geq 0, \quad\deg(x_i)=m\geq 1
\end{equation*}
for some $x_i\in V$ such that
\begin{equation*}
\left(\lambda\psi(x_i)\right)^{p-1}\left(\log \left(\lambda\psi(x_i)\right)\right)^q>C(V)m,\quad \lambda\psi(x_i)\geq 1,
\end{equation*}
then the lifespan of solution of the equation \eqref{equ} is finite. Moreover, we have
\begin{equation*}
\lim\limits_{\lambda\rightarrow+\infty}
(\lambda\psi(x_i))^{p-1}
(\log(\lambda\psi(x_i)))^{q}T_\lambda\leq\frac{1}{p-1}\leq\lim\limits_{\lambda\rightarrow+\infty}
(\lambda\|\psi\|_{\infty} )^{p-1}
(\log(\lambda\|\psi\|_{\infty}))^{q}T_\lambda.
\end{equation*}
\end{theorem}

\begin{theorem}\label{thm3}
Assume $p>1$, $q\geq1$ and $u$ is a positive solution to \eqref{equ}. Let
\begin{equation*}
D_{\mu}<\infty,\qquad
\eta(0)\ge1
\end{equation*}
where $\eta(0)$ is defined in \eqref{1***}.
Then there exists a constant
$C=C(p,q,V)>0$ such that
$T_\lambda<\infty$ whenever $\lambda>\Lambda$ and $\Lambda$ satisfies
\begin{equation*}
\Lambda\cdot\inf_{x\in V}\psi\geq \max\{ C(p,q,V ), 1 \}.
\end{equation*}
Furthermore,
\begin{equation*}
\lim\limits_{\lambda\rightarrow+\infty}
(\lambda\|\psi\|_{\infty} )^{p-1}
(\log(\lambda\|\psi\|_{\infty}))^{q}T_\lambda
=\frac{1}{p-1}.
\end{equation*}
\end{theorem}

\section{The proof of Theorem \ref{thm1}}
For the Cauchy equation \eqref{equ}, we define the upper and lower solution on infinite locally finite graph.

\begin{defn}
Assume
$
u\in L^{\infty}(V\times(0,T])
\cap C_V([0,T])
\cap C_V^{1}((0,T]).
$
If $u$ satisfies
\[
\begin{cases}
u_t-\Delta u\geq(\leq)f(x,t,u), & (x,t)\in V\times(0,T],\\
u(x,0)\geq(\leq)\lambda\psi, & x\in V,
\end{cases}
\]
then $u$ is called an upper (a lower) solution of \eqref{equ} with $u^p|\log u|^q$ replaced by $f(x,t,u)$.
\end{defn}

Before proving Theorem \ref{thm1}, we introduce
the maximum principle and the comparison principle on graphs. The proofs of Lemmas \ref{premax} and \ref{max} can be found in \cite[Lemmas 3.1-3.2]{cvhuyuanyang}.

\begin{lemma}\label{premax}
Let $\Omega\subset V$ be a finite connected graph and $z$ satisfy the equation
\begin{equation*}
z_t-\Delta z-cz\geq 0\quad \text{in}
\quad (\Omega\backslash\partial\Omega)\times(t_0,T],
\end{equation*}
where $c(x,t)<0$ in $(\Omega\backslash\partial\Omega)\times(t_0,T]$.
If $\inf\limits_{\Omega\times[t_0,T]}z<0$, then $z$ cannot reach its minimum in
$(\Omega\backslash\partial\Omega)\times(t_0,T].$

\end{lemma}

We now introduce a discrete Phragm\'{e}n--Lindel\"{o}f principle for parabolic equations on infinite locally finite graphs, whose proof can be found in \cite[Lemma 3.2]{cvhuyuanyang}.

\begin{lemma}\label{max}
Let $D_{\mu}\leq\infty$ and $T\in(0,+\infty]$.
Assume $c,z\in L^{\infty}(V\times I)$ for any bounded interval $I\subset[0,T)$ and $z\in C_V^1((0,T))\cap C_V^0([0,T))$. If $z$ satisfies
\begin{equation*}
\left\{
\begin{array}{lll}
z_t-\Delta z-cz\geq0,\ \ (x,t)\in V\times(0,T)\\[20pt]
z(x,0)\geq0,\ \ x\in V,
\end{array}
\right.
\end{equation*}
then $z\geq0$ on $V\times(0,T).$
\end{lemma}

The following lemma is a direct consequence of Lemmas \ref{premax} and \ref{max}.

\begin{lemma}\label{convex}
Let $D_{\mu}<\infty$ and $T\in(0,\infty]$.
Suppose $z\in L^{\infty}(V\times I)$ for any bounded interval $I\subset [0,T)$ and $z\in C^1_V((0,T))\cap C_V^0([0,T))$. If $z$ satisfies
the equation
\begin{equation*}
\left\{
\begin{array}{lll}
z_t-\Delta z\geq0,\ (x,t)\in V\times(0,T)\\[10pt]
z(x,0)=z_0,\ x\in V,
\end{array}
\right.
\end{equation*}
then we have
\begin{equation*}
z\geq\inf\limits_{x\in V} z_0(x).
\end{equation*}

\begin{proof}
Let $v=z-\inf\limits_{x\in V} z_0(x)$. Then $v_t$ satisfies the equation
\begin{equation*}
\left\{
\begin{array}{lll}
v_t-\Delta v\geq0,\ (x,t)\in V\times(0,T)\\[10pt]
v(x,0)=z_0-\inf\limits_{x\in V}z_0,\ x\in V,
\end{array}
\right.
\end{equation*}
According to Lemma \ref{max}, we have $v\geq0.$ Therefore, we obtain $z\geq\inf\limits_{x\in V} z_0(x)$.
\end{proof}
\end{lemma}

The following lemma compares the blow-up times of a upper solution and a low solution satisfying the same initial condition. A similar idea can be found in \cite[Lemma 4.2]{cvhuyuanyang}.

\begin{lemma}\label{lemcomparision}
Let $D_{\mu}<\infty$, $p\geq1$, $q\geq1$,
$t_i>0$, $u_i\in C_V([0,t_i))\cap C_V^1((0,t_i))$  and $u_i\in L^{\infty}(V\times[0,t_i'])$ for any given $0<t_i'<t_i$, $i=1,2$. Suppose $\lim\limits_{t\rightarrow t_i}
\|u_i(\cdot,t)\|_{l^{\infty}(V)}=\infty$, for $i=1,2$. If $u_i$, respectively, satisfies
\begin{equation}\label{014}
\left\{
\begin{array}{lll}
(u_1)_t-\Delta u_1 \geq u^p_1|\log u_1|^q,\ \ (x,t)\in V\times[0,t_1),\\[10pt]
u_1(x,0)\geq\psi(x),\ \  x\in V,
\end{array}
\right.
\end{equation}
and
\begin{equation}\label{015}
\left\{
\begin{array}{lll}
(u_2)_t-\Delta u_2 \leq u^p_2|\log u_2|^q,\ \ (x,t)\in V\times[0,t_2),\\[10pt]
u_2(x,0)\leq\psi(x),\ \  x\in V,
\end{array}
\right.
\end{equation}
then $t_1\leq t_2$.
\end{lemma}

\begin{proof}
If $t_1> t_2$, it follows from \eqref{014}--\eqref{015} that
\begin{equation*}
\left\{
\begin{array}{lll}
(u_1-u_2)_t-\Delta(u_1-u_2)
-\left(
u^p_1|\log u_1|^q-u^p_2|\log u_2|^q
\right)\geq0, \ (x,t)\in\ V\times[0,t_2),\\[10pt]
u_1(x,0)-u_2(x,0)\geq0,\ \  x\in V.
\end{array}
\right.
\end{equation*}
Denote $f(x)=x^p|\log x|^q$, for $x>0$, then
$$u^p_1|\log u_1|^q-u^p_2|\log u_2|^q=(u_1-u_2)f'(\xi),\quad \text{some }\xi\in(u_1,u_2).$$
By Lemma \ref{max},  $u_1-u_2\geq0$ on $V\times[0,t_2)$. Moreover,
$\lim\limits_{t\rightarrow t_2}
\|u_2(\cdot,t)\|_{l^{\infty}(V)}=\infty$, then
we obtain $\lim\limits_{t\rightarrow t_2}
\|u_1(\cdot,t)\|_{l^{\infty}(V)}=\infty$. But this is impossible because of $u_1\in L^{\infty}(V\times[0,t_2))$.

\end{proof}

The above lemma shows that the blow-up time of a upper solution cannot exceed that of a corresponding lower solution. This comparison property of blow-up times will be used in the sequel.

\begin{lemma}\label{ela0}Suppose $\Omega$ is a finite connected set in $V$. Let $\lambda_1$ be the smallest eigenvalue
of $-\Delta_{\Omega}$ and $\phi>0$ be the normalized corresponding eigenfunction.
The function  $\psi$ satisfies $\Lambda\inf\limits_{x\in V}\psi(x)\geq e^{1-q}$ and $\displaystyle\int_{\Omega}\lambda \psi\phi d\mu\geq1$, where $\Lambda$ is defined in Theorem \ref{thm1}.
If the equation \eqref{equ} has a global solution, then we have
\begin{equation*}
\left|\displaystyle\log\int_ {\Omega}\lambda\psi\phi d\mu\right|^q\leq\lambda_1.
\end{equation*}
\end{lemma}

\begin{proof}
For otherwise, if $\left|\displaystyle\log\int_ {\Omega}\lambda\psi\phi d\mu\right|^q>\lambda_1$, and $u$ is a global solution of \eqref{equ}, 
then we define
\begin{equation}\label{1***}
\eta(t)=\displaystyle\int_{\Omega} u(\cdot,t)\phi(\cdot)d\mu,\ \ t\geq0,
\end{equation}
it follows that $\left|\displaystyle\log
\eta(0)\right|^q=\left(\displaystyle\log
\eta(0)\right)^q>\lambda_1$, due to $\eta(0)=\displaystyle\int_{\Omega}\lambda \psi\phi d\mu\geq 1$. Hence, there exists a constant $\epsilon>0$ such that
\begin{equation}\label{1}
\left(\displaystyle\log\eta(0)\right)^q
=
(1+\epsilon)\lambda_1.
\end{equation}
We further compute that
\begin{equation}\label{05}
\begin{array}{lll}
\eta'(t)
&=\displaystyle\int_{\Omega}u_t\phi d\mu\\[20pt]
&=\displaystyle\int_{\Omega}\phi\Delta ud\mu
+\displaystyle\int_{\Omega} u |\log u|^q\phi d\mu.
\end{array}
\end{equation}
Next, we show that
\begin{equation*}
\phi(x)\Delta u(x,t)
\geq\phi(x)\Delta_{\Omega} u(x,t)
\ \ \text{for}\ \ x\in\Omega, \ \ t\geq0.
\end{equation*}
Indeed, for any $y\in V$ and $t\geq0$,
we have $u(y,t)\geq \zeta_{\Omega}u(y,t)$. This combined with $\phi(x)>0$ on $\Omega$ leads to
\begin{equation}\label{06}
\begin{array}{lll}
\phi(x)\Delta u(x,t)
&=\phi(x)\sum\limits_{y\in V:y\sim x}
\left(u(y,t)-u(x,t)\right)\displaystyle\frac{w_{xy}}{\mu(x)}\\[20pt]
&\geq\phi(x)\sum\limits_{y\in V:y\sim x}\left(\zeta_{\Omega}u(y,t)-u(x,t)\right)\displaystyle\frac{w_{xy}}{\mu(x)}\\[20pt]
&=\phi(x)\Delta_{\Omega}u(x,t).
\end{array}
\end{equation}
By Lemma \ref{convex}, $u\geq\lambda\inf\limits_{x\in V}\psi(x)\geq e^{1-q}$. It is easy to check that the function $x|\log x|^q$ is convex in $(e^{1-q},+\infty)$. Next,
putting the inequality \eqref{06} into \eqref{05}
and using Jensen's inequality imply
\begin{equation}\label{07}
\begin{array}{lll}
\eta'(t)
&=\displaystyle\int_{\Omega}\phi\Delta ud\mu
+\displaystyle\int_{\Omega} u |\log u|^q\phi d\mu\\[20pt]
&\geq\displaystyle\int_{\Omega}\phi\Delta_{\Omega}
ud\mu
+\displaystyle\int_{\Omega} u |\log u|^q\phi d\mu\\[20pt]
&\geq-\lambda_1
\displaystyle\int_{\Omega}u\phi d\mu
+\left(\displaystyle\int_V u\phi d\mu\right)\left|\log\displaystyle\int_V u\phi d\mu\right|^q\\[20pt]
&=-\lambda_1\eta(t)+\eta(t)\left|\log\eta(t)\right|^q.
\end{array}
\end{equation}
According to the continuity of $\eta(t)$ and \eqref{1},  there exists
a constant $\delta>0$ such that if $t\in(0,\delta)$, then $\eta'(t)>0$ for $t\in (0,\delta)$.
Next, we prove $\delta=+\infty$. For otherwise,
there exists $t_1\in(0,+\infty)$ such that $\eta'(t_1)=0$ and $\eta'(t_1)>0$ for any $t\in(0,t_1)$. Therefore, by \eqref{07}, we have
$(\log\eta(0))^q<(\log\eta(t_1))^q$ and
\begin{equation*}\label{08}
0=\eta'(t_1)\geq\eta(t_1)\left((\log\eta(t_1))^q-\lambda_1 \right)>0,
\end{equation*}
which is a contradiction. Hence, $\eta'(t)>0$ in $(0,+\infty)$. Then we obtain $\eta(t)>\eta(0)\geq1$ and
\begin{equation*}\label{010}
\left(\log\eta(t)\right)^q
>\left(\log\eta(0)\right)^q
=(1+\epsilon)\lambda_1,
\end{equation*}
which implies
\begin{equation}\label{011}
\lambda_1<\frac{1}{1+\epsilon}
\left(\log\eta(t)\right)^q.
\end{equation}
Using \eqref{07} together with \eqref{011}, we obtain
\begin{equation*}\label{012}
\eta'(t)\geq\eta(t)\cdot\frac{\epsilon}{1+\epsilon}
\cdot\left(\log\eta(t)\right)^q.
\end{equation*}
Consequently,
\begin{equation*}\label{013}
t\leq\frac{1+\epsilon}{(q-1)\epsilon}
\left(\log\eta(0)\right)^{1-q}.
\end{equation*}
Therefore,
\begin{equation*}
T_{\lambda}<+\infty.
\end{equation*}
However, this is impossible as $u$ is a global solution.

\end{proof}

The following lemma gives the asymptotic behavior of the lifespan $T_{\lambda}$
of the solution to \eqref{equ} as $\lambda\rightarrow +\infty$. It is a direct consequence of Lemma \ref{ela0}.

\begin{lemma}\label{lim}
Supposing $D_{\mu}<\infty$ and under the conditions in Lemma \ref{ela0}, there exists a constant $\Lambda$, as defined in Theorem \ref{thm1}, such that for any $\lambda>\Lambda,$
\begin{equation*}
T_{\lambda}<+\infty
\end{equation*}
and
\begin{equation*}
\lim\limits_{\lambda\rightarrow\infty}
(\log(\lambda\|\psi\|_{\infty}))^{q-1}T_{\lambda}=\frac{1}{q-1}.
\end{equation*}
\end{lemma}

\begin{proof}
For $0<T_0<\displaystyle
\frac{1}{q-1}(\log\lambda\|\psi\|_{\infty})^{1-q}$, we denote
\begin{equation*}
\bar{u}(t)
=\exp\left\{[
 (\log(\lambda\|\psi\|_{\infty}))^{1-q}-(q-1)t
]^{ \frac{-1}{q-1} }\right\}.
\end{equation*}
Then $\bar{u}$ is an upper solution of the equation \eqref{equ} with $p=1$ and $q>1$ in $V\times[0,T_0]$. Therefore, we have
\begin{equation*}
T_{\lambda}\geq\displaystyle
\frac{1}{q-1}\left(\log(\lambda\|\psi\|_{\infty})\right)^{1-q},
\end{equation*}
that is
\begin{equation}\label{016}
\left(\displaystyle\log(\lambda\|\psi\|_{\infty})\right)^{q-1}T_{\lambda}\geq\displaystyle
\frac{1}{q-1}.
\end{equation}
We next establish
\begin{equation*}
(\log\lambda \|\psi\|_{\infty})^{q-1}T_{\lambda}\leq\frac{1}{q-1}.
\end{equation*}
By the lemma \ref{ela0}, if
\begin{equation*}
\lambda>\exp\{\lambda_1^{\frac{1}{q}}\}\left(\int_ {\Omega}\psi\phi d\mu\right)^{-1}.
\end{equation*}
then  $T_{\lambda}<+\infty$, where $\phi$
and $\lambda_1$ are defined therein.
In fact, in Theorem \ref{thm1}, we can define 
$$\Lambda_1=:\sup_{\text{finite } \Omega\subset V}\exp\{\lambda_1^{\frac{1}{q}}\}\left(\int_ {\Omega}\psi\phi d\mu\right)^{-1}.$$
Although it seems that the constant $\Lambda_1$ does not coincide with the expressions $\Lambda$ in Theorem \ref{thm1} at this point, we will explain the coincidence at the end of the proof.
Motivated by the proof of the lemma \ref{ela0}, we introduce the continuous function $\eta(t)$,
\begin{equation*}
\eta(t)=\displaystyle\int_{\Omega} u(\cdot,t)\phi(\cdot)d\mu,\quad t\in[0,T),
\end{equation*}
which satisfies
\begin{equation*}
\eta'(t)\geq\eta(t)\left((\log\eta(t))^q-\lambda_1 \right)>0,\hspace{10pt} \eta(0)=\int_ {\Omega}\lambda\psi\phi d\mu.
\end{equation*}
Proceeding as in the proof of Lemma \ref{ela0}, we obtain that
\begin{equation*}
\eta(t)\ \text{is monotonically increasing for }t>0,
\end{equation*}
and
\begin{equation*}
\eta'(t)\geq\eta(t)\cdot\frac{\epsilon}{1+\epsilon}
\cdot\left(\log\eta(t)\right)^q.
\end{equation*}
Hence, we have
\begin{equation}\label{017}
T_{\lambda}
\leq\frac{1+\epsilon}{(q-1)\epsilon}
\left(\log\eta(0)\right)^{1-q},
\end{equation}
where
\begin{equation*}\label{018}
\epsilon
=\displaystyle\frac{(\log\eta(0))^q}{\lambda_1}-1
=
\displaystyle\frac{(\log\int_V\lambda\psi\phi d\mu)^q}
{\lambda_1}-1\rightarrow+\infty,\quad \text{as }\lambda\rightarrow +\infty.
\end{equation*}
Now, \eqref{017} implies that 
\begin{equation*}
\lim\limits_{\lambda\rightarrow +\infty}
\left(\displaystyle\log\int_ {\Omega}\lambda\psi\phi d\mu\right)^{q-1}T_{\lambda}\leq\frac{1}{q-1}.
\end{equation*}
Fix an arbitrary $\tilde{x}\in V$ and let $\Omega=\{\tilde{x}\}$. Since  $\int_{\Omega}\phi\,d\mu=1$, we obtain
$\phi(\tilde{x})\mu(\tilde{x})=1$, and the inequality above implies 
\begin{equation*}
\lim\limits_{\lambda\rightarrow +\infty}
\left(\log(\lambda\psi(\tilde{x})\right)^{q-1}T_{\lambda}\leq\frac{1}{q-1}.
\end{equation*}
Moreover, the arbitrary choice of $\tilde{x}$ actually implies that 
\begin{equation}\label{019*}
\lim\limits_{\lambda\rightarrow +\infty}
\left(\log\left(\lambda\|\psi\|_{\infty}\right)\right)^{q-1}T_{\lambda}\leq\frac{1}{q-1}.
\end{equation}
Now, it directly follows from \eqref{016} and \eqref{019*} that
\begin{equation*}
\lim\limits_{\lambda\rightarrow +\infty}
(\log\lambda\|\psi\|_{\infty})^{q-1}T_{\lambda}=\frac{1}{q-1}.
\end{equation*}
\end{proof}
\noindent
At the end of this section, we explain the choice of the constant $\Lambda$ in Theorem \ref{thm1}. The condition $\Lambda\cdot\inf_{x\in V}\psi\geq e^{1-q}$ guarantees the usage of Jensen's inequality, since the function $x|\log x|^q$ is convex in $(e^{1-q},+\infty)$. Now we turn to the other condition $\Lambda\cdot\inf_{x\in V}\psi\geq C(q,V )$. In the proof of Lemma \ref{lim}, we actually need 
$$\lambda\geq \Lambda_1\quad \text{with }\Lambda_1=:\sup_{\text{finite } \Omega\subset V}\exp\{\lambda_1^{\frac{1}{q}}\}\left(\int_ {\Omega}\psi\phi d\mu\right)^{-1},$$
where $\Omega$ is a set of single point. Now, for any $x\in V$, the definition of $\Lambda_1$ actually implies 
$$\Lambda_1\int_ {\Omega}\psi\phi d\mu=\Lambda_1 \psi(x)\geq C(q,V),$$
which coincides with the definition of $\Lambda$ in Theorem \ref{thm1}. This completes the proof of Theorem \ref{thm1}.

\section{The proof of Theorem \ref{thm2}}
\noindent
For convenience, we assume $w_{xy}=w_{yx}=1$ and $\mu(x)=1$, then the constant $C(V)=1$ in Theorem \ref{thm2}.
Without loss of generality, we assume $i=1$ in the statement of Theorem \ref{thm2}. Now recalling
$u(x_1,0)=\lambda\psi(x_1)>\exp\left\{m^{1/q}\right\}$, a direct computation yields that
\begin{equation*}
\begin{array}{lll}
&\Delta u(x_1,0) + u(x_1,0)\big|\log
u(x_1,0)\big|^q\\[10pt]
=&\Delta u(x_1,0) + u(x_1,0)\big(\log
u(x_1,0)\big)^q\\[10pt]
=&\displaystyle\sum\limits_{y\sim x_1}(u(y,0)-u(x_1,0))
+u(x_1,0)\big(\log u(x_1,0)\big)^q\\[10pt]
\geq&-mu(x_1,0)+u(x_1,0)\big(\log u(x_1,0)\big)^q\\[10pt]
>&0,
\end{array}
\end{equation*}
which, by equation \eqref{equ} and the continuity of the solution $u$ in time, shows that $u_t(x_1,t)>0$ for sufficiently small $t>0$. Next, we show that $u_t(x_1,t)$ is monotonically increasing on $(0,T_{\lambda})$.
Let $\delta$ be the first time such that
\begin{equation*}
u_t(x_1,t)=0,
\end{equation*}
then
\begin{equation*}
u_t(x_1,\delta)>0,\quad t\in(0,\delta).
\end{equation*}
Therefore, $$u(x_1,\delta)>u(x_1,0)>\exp\left\{m^{1/q}\right\},\quad \log u(x_1,\delta)> \log u(x_1,0)>m^{1/q}.$$
A direct computation yields
\begin{equation}\label{**}
\begin{array}{lll}
u_t(x_1,\delta)
&=\Delta u(x_1,\delta) + u(x_1,\delta)\big(\log
u(x_1,\delta)\big)^q\\[10pt]
&=\displaystyle\sum\limits_{y\sim x_1}(u(y,\delta)-u(x_1,\delta))
+u(x_1,\delta)\big(\log u(x_1,\delta)\big)^q\\[10pt]
&\geq-mu(x_1,\delta)+u(x_1,\delta)\big(\log u(x_1,\delta)\big)^q\\[10pt]
&>0,
\end{array}
\end{equation}
which leads to a contradiction.
Therefore, we obtain that $u(x_1,t)$ is monotonically increasing in $(0,T_{\lambda})$.
For any $t\in (0,T_{\lambda})$, we have
\begin{equation}\label{020}
\big(\log u(x_1,t)\big)^q> \big(\log u(x_1,0)\big)^q
> m,
\end{equation}
then there exists a constant $\epsilon>0$ such that
\begin{equation}\label{021}
\big(\log u(x_1,0)\big)^q
=m(1+\epsilon)
\end{equation}
and
\begin{equation}\label{022}
\big(\log u(x_1,t)\big)^q\geq
m(1+\epsilon).
\end{equation}
Using \eqref{**} for $\delta=t$, there holds
\begin{equation*}
u_t(x_1,t)
\geq-mu(x_1,t)+u(x_1,t)\big(\log u(x_1,t)\big)^q.
\end{equation*}
Now, according to \eqref{022} and the inequality above,
we obtain
\begin{equation}\label{023}
u_t(x_1,t)\geq\frac{\epsilon}{1+\epsilon}
u(x_1,t)\big(\log u(x_1,t)\big)^q.
\end{equation}
It follows from inequality \eqref{023} that
\begin{equation}\label{024}
\begin{array}{lll}
t&\leq\displaystyle\frac{1+\epsilon}{(q-1)\epsilon}
\big(\log u(x_1,0)\big)^{1-q}\\[20pt]
&=
\displaystyle\frac{1+\epsilon}{(q-1)\epsilon}
\big(\log(\lambda\psi(x_1))\big)^{1-q}.
\end{array}
\end{equation}
Hence,
\begin{equation}\label{025}
T_{\lambda}
\leq
\displaystyle\frac{1+\epsilon}{(q-1)\epsilon}
\big(\log(\lambda\psi(x_1))\big)^{1-q}.
\end{equation}
By \eqref{021}, we have
\begin{equation}\label{026}
\epsilon=\displaystyle
\frac{\left(\log\lambda+\log\psi(x_1)\right)^q}{m}-1.
\end{equation}
Combining \eqref{025}-\eqref{026} yields
\begin{equation}\label{027}
\lim\limits_{\lambda\rightarrow+\infty}(\log(\lambda\psi(x_1)))^{q-1}T_{\lambda}
\leq\displaystyle\frac{1}{q-1}.
\end{equation}
On the other hand,
for $0<t<\displaystyle
\frac{1}{q-1}(\log(\lambda\|\psi\|_{\infty}))^{1-q}$, let
\begin{equation*}
\bar{u}(t)
=\exp\left\{
 (\log(\lambda\|\psi\|_{\infty}))^{1-q}-(q-1)t
\right\}^{ \frac{-1}{q-1} }.
\end{equation*}
It is easy to check that $\bar{u}(t)$ is an upper solution to the equation \eqref{equ} with $p=1$ and $q>1$.
 Then, it follows from lemma \ref{lemcomparision} that
\begin{equation*}
T_{\lambda}\geq\displaystyle
\frac{1}{q-1}(\log(\lambda\|\psi\|_{\infty}))^{1-q},
\end{equation*}
namely,
\begin{equation*}\label{029}
(\displaystyle\log(\lambda\|\psi\|_{\infty}))^{q-1}T_{\lambda}\geq\displaystyle
\frac{1}{q-1}.
\end{equation*}
The inequality above together with \eqref{027} yields the desired estimates in Theorem \ref{thm2}.


\section{The proof of Theorem \ref{thm4} and \ref{thm3}}

\subsection{The proof of Theorem \ref{thm4}}
\noindent
For convenience, we assume $w_{xy}=w_{yx}=1$ and $\mu(x)=1$, then the constant $C(V)=1$ in Theorem \ref{thm4}. Without loss of generality, we assume $i=1$ in Theorem \ref{thm4} and recall
\begin{equation}\label{031}
\left(\lambda\psi(x_1)\right)^{p-1}\left|\log \left(\lambda\psi(x_1)\right)\right|^q>m,\quad \lambda\psi(x_i)\geq 1.
\end{equation}
Then, a direct calculation implies
\begin{equation}\label{11**}
\begin{array}{lll}
&\Delta u(x_1,0) + u^p(x_1,0)\big|\log
u(x_1,0)\big|^q\\[10pt]
=&\Delta u(x_1,0) + u^p(x_1,0)\big(\log
u(x_1,0)\big)^q\\[10pt]
=&\displaystyle\sum\limits_{y\sim x_1}(u(y,0)-u(x_1,0))
+u^p(x_1,0)\big(\log u(x_1,0)\big)^q\\[10pt]
\geq&-mu(x_1,0)+u^p(x_1,0)\big(\log u(x_1,0)\big)^q\\[10pt]
>&0,
\end{array}
\end{equation}
which, by equation \eqref{equ} and the continuity of the solution $u$ in time, shows that $u_t(x_1,t)>0$ for sufficiently small $t>0$. Then similar to proof the inequality \eqref{**}, we obtain that $u(x_1,t)$ is monotonically increasing in $(0,T_{\lambda})$, which combined with \eqref{031},
implies
\begin{equation}\label{032}
u^{p-1}(x_1,t)\big(\log u(x_1,t)\big)^q
>
u^{p-1}(x_1,0)\big(\log u(x_1,0)\big)^q
=m(1+\epsilon),\quad t\in (0,T_{\lambda}).
\end{equation}
Now, it is easy to check that 
\begin{equation}\label{22**}
1+\epsilon=\frac{u^{p-1}(x_1,0)\big(\log u(x_1,0)\big)^q}{m}=\frac{(\lambda\psi)^{p-1}(x_1)\big(\log (\lambda\psi(x_1))\big)^q}{m}.
\end{equation}
Then,
\begin{equation}\label{033}
m<\frac{1}{1+\epsilon}u^{p-1}(x_1,t)\big(\log u(x_1,t)\big)^q.
\end{equation}
Next, similar as the proof of \eqref{11**}, we have
\begin{equation*}
u_t(x_1,t)\geq
-mu(x_1,t)+u^p(x_1,t)\big(\log u(x_1,t)\big)^q,
\end{equation*}
which combined with \eqref{033} gives
\begin{equation*}
u_t(x_1,t)\geq\frac{\epsilon}{1+\epsilon}
u^p(x_1,t)\big(\log u(x_1,t)\big)^q.
\end{equation*}
Integrating the above inequality over $(0,t)$ yields
\begin{equation}\label{034}
\displaystyle\int_{0}^{t}
\frac{u_t(x_1,t)}{u^p(x_1,t)\big(\log u(x_1,t)\big)^q}dt\geq\frac{\epsilon}{1+\epsilon}t.
\end{equation}
To investigate the property of the inequality above, we define
\begin{equation*}
F(r)=:\int_{r}^{+\infty}\frac{ds}{s^p(\log s)^q}\quad \text{for }r>0.
\end{equation*}
One can easily check that
\begin{equation}\label{035}
\frac{d}{dt}F(u(x_1,t))
=
\frac{-u_t(x_1,t)}{u^p(x_1,t)(\log u(x_1,t))^q }.
\end{equation}
It follows from \eqref{034} and \eqref{035} that
\begin{equation}\label{036}
\begin{array}{lll}
t&\leq
\displaystyle\frac{1+\epsilon}{\epsilon}
\left(-F(u(x_1,t))+F(u(x_1,0))\right)\\[12pt]
&\leq
\displaystyle\frac{1+\epsilon}{\epsilon}F(u(x_1,0))\\[12pt]
&=
\displaystyle\frac{1+\epsilon}{\epsilon}F(\lambda\psi(x_1)).
\end{array}
\end{equation}
By letting $s=\lambda\psi(x_1)y$, we now compute $F(\lambda\psi(x_1))$ as follows:
\begin{equation}\label{037}
\begin{array}{lll}
F(\lambda\psi(x_1))
&=
\displaystyle\int_{\lambda\psi(x_1)}^{+\infty}
\frac{ds}{s^p(\log s)^q}\\[15pt]
&=
\displaystyle\int_{1}^{+\infty}
\frac{dy}{y^p(\lambda\psi(x_1))^{p-1}
\{\log(\lambda\psi(x_1))+\log y \}^q }\\[15pt]
&\leq
\displaystyle\frac{1}{p-1}\cdot\frac{1}{ (\lambda\psi(x_1))^{p-1}
\{\log(\lambda\psi(x_1))\}^q},
\end{array}
\end{equation}
which, together with \eqref{036}, yields
\begin{equation*}
T_{\lambda}
\leq
\displaystyle\frac{1+\epsilon}{\epsilon}\cdot
\displaystyle\frac{1}{p-1}\cdot\frac{1}{ (\lambda\psi(x_1))^{p-1}
\{\log(\lambda\psi(x_1))\}^q  }.
\end{equation*}
The inequality above combined with \eqref{22**} implies that
\begin{equation}\label{038}
\lim\limits_{\lambda\rightarrow+\infty}
(\lambda\psi(x_1) )^{p-1}
(\log\lambda\psi(x_1))^{q}T_\lambda
\leq
\frac{1}{p-1}.
\end{equation}

\noindent
It remains to prove the other inequality in Theorem \ref{thm4}. Note that $F$ is a continuous decreasing function in $(0,+\infty)$, so its inverse exists, denoted by $F^{-1}$. We now define
\begin{equation}\label{123}
\overline{u}(t)=:F^{-1}(F(\lambda\|\psi\|_{\infty})-t)
\end{equation}
on the interval $0\leq t<F(\lambda\|\psi\|_{\infty})$. Clearly, $\overline{u}(0)=\lambda\|\psi\|_{\infty}$ and we note that $F(\overline{u}(t) )\rightarrow0$
as $t\rightarrow F(\lambda\|\psi\|_{\infty})$, then the definition
of $F$ implies
\begin{equation*}
\overline{u}(t)\rightarrow+\infty
\quad\text{as}\quad t\rightarrow F(\lambda\|\psi\|_{\infty}).
\end{equation*}
It follows from $F(\overline{u}(t))
=F(\lambda\|\psi\|_{\infty})-t $ and \eqref{035}
that $\bar{u}$ satisfies the following equation
\begin{equation*}
\partial_t\overline{u}(t)
=
\Delta \overline{u}(t)+\overline{u}^p(t)(\log\overline{u}(t))^q,
\end{equation*}
with $\overline{u}(0)=\lambda\|\psi\|_{\infty}$.
Hence, $\overline{u}(t)$ is an upper solution to equation \eqref{equ} in
$V\times [0,T)$. Then by the lemma \ref{lemcomparision},
\begin{equation*}
T_{\lambda}\geq F(\lambda\|\psi\|_{\infty}).
\end{equation*}
Using the same argument as in the proof of \eqref{036}, we obtain
\begin{equation*}
\begin{array}{lll}
F(\lambda\|\psi\|_{\infty})
&=
\displaystyle\int_{\lambda\|\psi\|_{\infty}}^{\infty}
\frac{ds}{s^p(\log s)^q}\\[15pt]
&=
\displaystyle\int_{1}^{+\infty}
\frac{dy}{y^p(\lambda\|\psi\|_{\infty})^{p-1}
\{\log(\lambda\|\psi\|_{\infty})+\log y \}^q }.
\end{array}
\end{equation*}
Then we get
\begin{equation*}
(\lambda\|\psi\|_{\infty} )^{p-1}
(\log\lambda\|\psi\|_{\infty})^{q}T_\lambda
\geq
\displaystyle\int_{1}^{+\infty}
\frac{dy}
{y^p
\left(
1+\frac{\log y}{\log(\lambda\|\psi\|_{\infty} )}
\right)^q },
\end{equation*}
which, by dominated convergence theorem, implies that
\begin{equation}\label{039}
\lim\limits_{\lambda\rightarrow+\infty}
(\lambda\|\psi\|_{\infty} )^{p-1}
(\log\lambda\|\psi\|_{\infty})^{q}T_\lambda
\geq
\frac{1}{p-1}.
\end{equation}
Combining \eqref{038} and \eqref{039}, we obtain the desired result in Theorem \ref{thm4}.

\subsection{The proof of Theorem \ref{thm3}}
In this subsection, we briefly explain the ideas to prove Theorem \ref{thm3}.
The proofs of the first part in Theorem \ref{thm3} (i.e., $T_\lambda<\infty$ whenever $\lambda>\Lambda$ and $\Lambda\cdot\inf_{x\in V}\psi\geq \max\{ C(p,q,V ), 1 \}$) and $$\lim\limits_{\lambda\rightarrow+\infty}
(\lambda\|\psi\|_{\infty} )^{p-1}
(\log(\lambda\|\psi\|_{\infty}))^{q}T_\lambda
\leq\frac{1}{p-1}$$ 
are totally similar to the proofs of Lemma \ref{ela0} and Lemma \ref{lim} after some minor modifications, so we omit the details.

It remains to show the opposite direction. Using the methods in the proofs of Theorems \ref{thm1}-\ref{thm4}, we need to construct an upper solution to equation \eqref{equ} in $V\times [0,T)$, which, actually, has been constructed in \eqref{123} and further implies 
$$\lim\limits_{\lambda\rightarrow+\infty}
(\lambda\|\psi\|_{\infty} )^{p-1}
(\log(\lambda\|\psi\|_{\infty}))^{q}T_\lambda
\geq\frac{1}{p-1}.$$
Consequently, the desired results in Theorem \ref{thm3} hold.


\section*{Acknowledgements}
Both authors gratefully acknowledge financial support from the China Scholarship Council (CSC).\\

\noindent\textbf{Conflicts of interest/Competing interests}: Not applicable.\\
\textbf{Datum availability statement}: Data sharing is not applicable to this article as no datasets were generated or analysed during the current study.\\

\normalem\bibliographystyle{plain}{}

\end{document}